\documentclass [a4paper,12pt]{article}
\usepackage{amsmath}
\usepackage{amsfonts}
\usepackage{indentfirst}
\usepackage{esint}
\usepackage{latexsym}
\usepackage{bbm}
\usepackage{amsfonts}
\usepackage{mathrsfs}
\usepackage{amssymb}
\usepackage{amsmath}
\usepackage{amsthm}
\usepackage{latexsym}
\usepackage{indentfirst}
\usepackage{enumitem}
\usepackage{bm}
\usepackage{esint}
\usepackage{hyperref}
\usepackage{cleveref}
\usepackage[numbers,square]{natbib}
\numberwithin{equation}{section}
\usepackage{color}
\allowdisplaybreaks
\usepackage{amsthm}
\usepackage{amssymb}
\usepackage{enumerate}
\usepackage{ulem}
\allowdisplaybreaks

\newtheorem{theorem}{Theorem}[section]
\newtheorem{lemma}[theorem]{Lemma}

\newtheorem{rmk}[theorem]{Remark}

\makeatletter
\usepackage{amsmath}
\usepackage{amsfonts}
\usepackage{indentfirst}
\usepackage{esint}
\usepackage{latexsym}
\usepackage{authblk}
\usepackage{color}
\UseRawInputEncoding
\usepackage{amsthm}
\usepackage{amssymb}
\usepackage{enumerate}
\usepackage{ulem}
\makeatletter

\newcommand{\Rmnum}[1]{\expandafter\@slowromancap\romannumeral #1@}
\makeatother
\begin{document}

\title{Sign-preserving solutions to the Tzitz\'eica equation on lattice graphs}
\author[a,b]{Pengxiu Yu\thanks{Email: Pxyu@ruc.edu.cn}}
\author[c,d]{Yiping Zhang\thanks{Corresponding author, Email: zhangyiping161@mails.ucas.ac.cn}}
\affil[a]{\footnotesize{School of Mathematics,
Renmin University of China, Beijing 100872, China}}
\affil[b]{\footnotesize{Fakult\"{a}t f\"{u}r Mathematik, Universit\"{a}t Bielefeld, Bielefeld 100131, Germany
}}
\affil[c]{\footnotesize{School of Mathematics and Statistics, and Hubei Key Laboratory of Mathematical Sciences, Central China Normal University, Wuhan 430079, China}}
\affil[d]{\footnotesize{Key Laboratory of Nonlinear Analysis \& Applications (Ministry of Education), Central China Normal University, Wuhan 430079, China}}
\date{}
\maketitle

\begin{abstract}
On the lattice graph $\mathbb{Z}^n$, we establish the existence of sign-preserving solutions to the Tzitz\'eica equation. We prove the existence of positive solutions and two classes of negative solutions under different assumptions, and derive their decay estimates. The proof is based on a suitable approximation scheme, the monotone convergence theorem, and an exhaustion argument. These results extend those of Hua, Huang, and Wang (Anal. PDE, 2026) by establishing the existence of both positive and negative sign-preserving solutions together with their decay estimates.
\end{abstract}

\section{Introduction and main results}\label{s1}

The Tzitz\'eica model
\begin{equation}\label{24}
u_{xy}=e^u-e^{-2u}
\end{equation}
originates from the pioneering work of G. Tzitz\'eica \cite{TZ1908,TZ1910,TZ3}. It first arose in the study of surfaces in $\mathbb{R}^3$ for which the ratio of the negative Gaussian curvature to the fourth power of the distance from the tangent plane to a fixed point is constant. The equation \eqref{24} also arises in the Euler equation for one-dimensional ideal gas dynamics \cite{Euler1,Euler2,Euler3} and, in magnetohydrodynamics, is equivalent to the Hirota--Satsuma partial differential equation \cite{hs1,hs2}. See also \cite{see1,see2} and the references therein.

Analysis on graphs has attracted considerable attention in recent years, not only because of its rich mathematical structure but also because of its broad applications in image processing, data mining, and complex networks \cite{neural1,neural2,neural3,neural4}. Among the many research directions, partial differential equations arising from geometry and physics have received considerable attention in the graph setting. A. Grigor'yan et al. \cite{griyang3,griyang2,griyang1} established the Sobolev spaces and the corresponding functional framework on graphs, laying the foundation for subsequent developments. Since then, extensive research has been devoted to partial differential equations on graphs. In particular, X. Zhang et al. \cite{pq} investigated the multiplicity and uniqueness of positive solutions to the superlinear singular $(p,q)$-Laplacian equation, while P. Yu \cite{ypxjmp} and L. Sun et al. \cite{sunlladvance} studied Kazdan--Warner type equations using the Brouwer degree method. We refer the reader to \cite{huakeller,huabbli,linyau,stz,sunsin} for further related results.

Recently, A. Jevnikar and W. Yang \cite{tzcv} studied blow-up phenomena for the Tzitz\'eica equation
\begin{equation}\label{jyangtz}
\Delta u+h_1e^u-h_2e^{-2u}=0,
\end{equation}
in the unit ball $\mathbb{B}_1\subset\mathbb{R}^2$, where $h_1$ and $h_2$ are smooth positive functions. When $h_2\equiv0$, equation \eqref{jyangtz} reduces to the classical Liouville equation
\begin{equation}\label{liouv}
\Delta u+h_1e^u=0,
\end{equation}
which arises in the study of Gaussian curvature under conformal changes of the metric; see \cite{liouv1,liouv2,liouv3,liouv4}. In the same work, they also established a sharp Moser--Trudinger type inequality on compact surfaces $M$ and derived quantization properties of local blow-up masses associated with blow-up sequences of solutions.

Motivated by the above works, in this paper we extend the Tzitz\'eica equation \eqref{jyangtz} to lattice graphs and establish several new results. Specifically, we study the following Tzitz\'eica equations.
\begin{equation}\label{tz11}
-\Delta u+\lambda_1e^{Au(x)}(e^{Au(x)}-1)
+\lambda_2e^{-Bu(x)}(e^{-Bu(x)}-1)\pm h_3(x)=0
\end{equation}
and
\begin{equation}\label{tz22}
-\Delta u+\lambda_1(e^{Au(x)}-1)
+\lambda_2(e^{-Bu(x)}-1)\pm h_3(x)=0,
\end{equation}
where
\begin{equation}\label{1.1}
h_3=4\pi\sum_{j=1}^{M}n_j\delta_{p_j},
\end{equation}
$n_1,\ldots,n_M$ are positive constants, and
$p_1,\ldots,p_M$ are distinct vortex points in $\mathbb{Z}^n$. $\lambda_1$, $\lambda_2$ and $A$, $B$ are positive constants to be determined.

When $\lambda_2\equiv0$ in \eqref{tz11}, we obtain the self-dual Chern-Simons vortex equation,
\begin{equation}\label{cs}
-\Delta u+\lambda_1e^{Au}(e^{Au}-1)+h_3=0.
\end{equation}
There have been many studies on this model since then, such as
\cite{houk,MR5045709,lisunyang,wangchuhua}.
If $\lambda_2\equiv0$, the Tzitz\'eica  equation \eqref{tz22} reduces to the Abelian Higgs equation,
\begin{equation}\label{ah}
-\Delta u+\lambda_1(e^{Au}-1)+h_3=0.
\end{equation}
There have been many studies of this model; see
\cite{ah3,ah1,ah2}.
The Tzitz\'eica equations \eqref{tz11}--\eqref{tz22} generalize both the Chern--Simons equation \eqref{cs} and the Abelian Higgs equation \eqref{ah}. Therefore, their analysis is considerably more challenging than that of either equation alone.

Before presenting our problem and main results, we introduce some basic notation for lattice graphs. Throughout this paper, let $n\ge2$ and
\begin{equation*}
V=\mathbb{Z}^n
=\left\{
x=(x_1,\ldots,x_n):
x_i\in\mathbb{Z},\ 1\le i\le n
\right\}.
\end{equation*}
The distance on $\mathbb{Z}^n$ is defined by
\begin{equation*}
d(x,y)=\sum_{i=1}^{n}|x_i-y_i|,
\qquad x,y\in\mathbb{Z}^n,
\end{equation*}
and we write $d(x)=d(x,0)$. The edge set is
\begin{equation*}
E=\left\{
\{x,y\}:x,y\in\mathbb{Z}^n,\;
\sum_{i=1}^{n}|x_i-y_i|=1
\right\}.
\end{equation*}
If $\{x,y\}\in E$, we write $x\sim y$.
We denote by
\begin{equation*}
C(\mathbb{Z}^n)
=\{u:\mathbb{Z}^n\rightarrow\mathbb{R}\}
\end{equation*}
the space of all real-valued functions on $\mathbb{Z}^n$, and define the support of $u$ by
\begin{equation*}
\operatorname{supp}(u)
=
\left\{
x\in\mathbb{Z}^n:u(x)\neq0
\right\}.
\end{equation*}
Let $C_0(\mathbb{Z}^n)$ denote the space of all real-valued functions with finite support.
For any finite subset $\Omega\subset\mathbb{Z}^n$, its boundary is defined by
\begin{equation*}
\partial\Omega=
\left\{
y\in\mathbb{Z}^n\setminus\Omega:
\exists\,x\in\Omega\ \text{such that}\ y\sim x
\right\},
\end{equation*}
and we write
\begin{equation*}
\overline{\Omega}=\Omega\cup\partial\Omega.
\end{equation*}
For any function $u:\mathbb{Z}^n\to\mathbb{R}$,
the $l^p$-norm of $u$ is defined by
\begin{equation*}
||u||_{l^p(\mathbb{Z}^n)}
=
\left\{
\begin{array}{lll}
\left(\sum_{x\in\mathbb{Z}^n}|u(x)|^p\right)^{\frac{1}{p}}, \ 1\leq p<\infty,\\[20pt]
\mathrm{sup}_{\mathbb{Z}^n}|u(x)|,\ \ p=\infty.
\end{array}
\right.
\end{equation*}
The difference operator is defined by
\begin{equation*}
\nabla_{xy}u=u(y)-u(x),
\quad x,y\in\mathbb{Z}^n.
\end{equation*}
The graph Laplacian is defined by
\begin{equation*}
\Delta u(x)=\sum_{d(x,y)=1}\bigl(u(y)-u(x)\bigr).
\end{equation*}

Our main results are as follows.

\begin{theorem}\label{positivethm1}
For $\lambda_1A-\lambda_2B\geq0$, the  Tzitz\'eica  equation
\begin{equation}\label{tz1}
-\Delta u+\lambda_1e^{Au(x)}(e^{Au(x)}-1)
+\lambda_2e^{-Bu(x)}(e^{-Bu(x)}-1)-h_3(x)=0
\end{equation}
admits a positive solution
$u\geq0$  in $\mathbb{Z}^n$. If $\lambda_1A-\lambda_2B>0$, the solution has the
decay estimate
\begin{equation*}
u=O(e^{-md(x)}),
\end{equation*}
where $m=\log\left(
1+\frac{\lambda_1A-\lambda_2B}{2n}
\right)$.
\end{theorem}

To better introduce Theorem \ref{negativethm2}, we first define
\begin{equation*}
g(x)=\lambda_1e^{Ax}(e^{Ax}-1)
+\lambda_2e^{-Bx}(e^{-Bx}-1).
\end{equation*}
For $\lambda_1A-\lambda_2B>0$, it is easy to see that $g(x)$ has exactly two zeros, $s_0<0$ and $0$, and that $g(x)<0$ for any $x\in(s_0,0)$.

\begin{theorem}\label{negativethm2}
For $\lambda_1A-\lambda_2B>0$ and any $\delta\in(0,-s_0)$, with $h_3$ defined by \eqref{1.1}, if $4\pi\max_j{n_j}\leq-g(s_0+\delta)$, then
there  exists a nonpositive solution $u\in[s_0+\delta,0]$ to the  Tzitz\'eica  equation
\begin{equation}\label{tz2}
-\Delta u+\lambda_1e^{Au(x)}(e^{Au(x)}-1)
+\lambda_2e^{-Bu(x)}(e^{-Bu(x)}-1)+h_3(x)=0
\end{equation}
in $\mathbb{Z}^n$. Moreover, the solution $u$ has a decay estimate
\begin{equation*}
u=O(e^{-m(1-\epsilon_0)d(x)}),
\end{equation*}
where $m=\log\left(
1+\frac{\lambda_1A-\lambda_2B}{2n}
\right)$ and $0<\epsilon_0<1$.

\end{theorem}
\begin{rmk}
In Theorem \ref{negativethm2}, we require
$4\pi\max_j n_j\leq-g(s_0+\delta)$. Conversely, if
$4\pi\max_j n_j$ is sufficiently large, then equation \eqref{tz2}
does not admit any nonpositive solution. Indeed, after relabeling if
necessary, we may assume that $n_1=\max_j n_j$. Suppose that $u$ is a
nonpositive solution of \eqref{tz2}. Since $u\leq 0$, evaluating
\eqref{tz2} at $p_1$ yields
$$
\begin{aligned}
0={}&-\Delta u(p_1)
+\lambda_1e^{Au(p_1)}(e^{Au(p_1)}-1)
+\lambda_2e^{-Bu(p_1)}(e^{-Bu(p_1)}-1)+h_3(p_1)\\[8pt]
\geq{}&
2n u(p_1)-\lambda_1
+\lambda_2e^{-Bu(p_1)}(e^{-Bu(p_1)}-1)
+4\pi n_1.
\end{aligned}
$$
It is easy to verify that the function $2nt-\lambda_1+\lambda_2e^{-Bt}(e^{-Bt}-1)$  is bounded from below on $(-\infty,0]$. Therefore,
if $n_1$ is suitably large, then
$$
2n u(p_1)-\lambda_1
+\lambda_2e^{-Bu(p_1)}(e^{-Bu(p_1)}-1)
+4\pi n_1>0,
$$
which contradicts the above equality. Hence, equation \eqref{tz2}
does not admit any nonpositive solution.
\end{rmk}


\begin{theorem}\label{positivetz2thm1}
For $\lambda_1A-\lambda_2B>0$,
there exists a unique bounded nonnegative solution $\tilde{u}$ to the
 Tzitz\'eica  equation,
\begin{equation}\label{tz3}
-\Delta u+\lambda_1(e^{Au}-1)
+\lambda_2(e^{-Bu}-1)-h_3=0,
\end{equation}
 satisfying  $0\leq u\leq\tilde{u}$, where $u$ is  obtained in Theorem
\ref{positivethm1}. Moreover, the solution $\tilde{u}$ has the decay estimate
\begin{equation*}
\tilde{u}=O(e^{-md(x)}),
\end{equation*}
where $m=\log\left(
1+\frac{\lambda_1A-\lambda_2B}{2n}
\right)$.
\end{theorem}

\begin{theorem}\label{nagetivetz2tm2}
For $\lambda_1A-\lambda_2B>0$ and any $\delta\in(0,-s_0)$, with $h_3$ defined by \eqref{1.1}, if $4\pi\max_j{n_j}\leq-g(s_0+\delta)$, then the equation
\begin{equation}\label{tz4}
-\Delta u+\lambda_1(e^{Au}-1)
+\lambda_2(e^{-Bu}-1)+h_3=0,
\end{equation}
 admits a solution $\bar{u}$ in $\mathbb{Z}^n$. Moreover,   $\bar{u}$ satisfies $0\geq\bar{u}\geq u$, where $u$ is obtained in Theorem \ref{negativethm2}. Furthermore,  $\bar{u}$ has the same decay estimate as $u$,
\begin{equation*}
\bar{u}=O(e^{-m(1-\epsilon_0)d(x)}),
\end{equation*}
where $m=\log\left(
1+\frac{\lambda_1A-\lambda_2B}{2n}
\right)$ and $0<\epsilon_0<1$.
\end{theorem}

\begin{rmk} Let $$ \mu:=\lambda_1A-\lambda_2B>0, \qquad m:=\log\left(1+\frac{\mu}{2n}\right). $$
The difference between the decay rates in
Theorems \ref{positivethm1} and \ref{positivetz2thm1} and those in Theorems \ref{negativethm2}
and \ref{nagetivetz2tm2} is due to the sign of the solutions. For the nonnegative solutions in Theorems \ref{positivethm1} and
\ref{positivetz2thm1}, the corresponding nonlinear terms satisfy
$$ \frac{\text{the corresponding nonlinear
term}}{t}\geq \mu, \qquad t>0. $$ Therefore,
outside a finite set, we obtain $$ \Delta u\geq
\mu u. $$ Hence the full linearized coefficient $\mu$ can be used in the comparison argument,
which gives $$ u(x)=O(e^{-m d(x)}). $$ Thus,
there is no loss in the exponential decay rate.
For the negative solutions in Theorems \ref{negativethm2}
and \ref{nagetivetz2tm2}, however, the corresponding nonlinear
quotient satisfies $$
\lim_{t\to0^-}\frac{\text{the corresponding nonlinear term}}{t}=\mu, $$ but approaches $\mu$ from below. Consequently, in general one cannot
use $\mu$ as a uniform lower bound on the
negative side. Instead, one has to choose a constant $c_{\epsilon_0}<\mu$, for example, $$ c_{\epsilon_0} = 2n\left[ \left(1+\frac{\mu}{2n}\right)^{1-\epsilon_0}-1 \right]. $$ The corresponding exponential rate is $$ \log\left(1+\frac{c_{\epsilon_0}}{2n}\right) = (1-\epsilon_0)m. $$ Therefore, the negative solutions satisfy $$ u(x)=O\left(e^{-m(1-\epsilon_0)d(x)}\right). $$ Hence, Theorems \ref{positivethm1} and \ref{positivetz2thm1} have no loss in the exponential decay rate, whereas Theorems \ref{negativethm2}
and \ref{nagetivetz2tm2} have a loss due to the behavior of the nonlinear terms on the negative side. \end{rmk}
\section{Preliminaries}
In this section, we recall the following maximum principle, which will be used
repeatedly throughout the paper.
\begin{lemma}[Maximum principle]\label{maxprinciple}
Let $\Omega\subset\mathbb Z^n$ be a finite subset, and let $c\in C(\bar\Omega)$ be positive. Suppose that $u\in C(\bar\Omega)$ satisfies
\begin{equation*}
\left\{
\begin{array}{lll}
(\Delta-c)u\geq0\ \text{on}\ \Omega,\\[10pt]
u\leq0\ \text{on}\ \partial\Omega.
\end{array}
\right.
\end{equation*}
Then $u\leq0$ on $\bar\Omega$.
\end{lemma}

\section{The proof of Theorem \ref{positivethm1}}
Let $u_0=0$ and choose a sufficiently large constant $L>0$. We consider the following iterative  equations,
\begin{equation}\label{iteraequa}
\left\{
\begin{array}{llll}
(\Delta-L)u_{k+1}
=\lambda_1e^{Au_{k}}(e^{Au_{k}}-1)
+\lambda_2e^{-Bu_{k}}(e^{-Bu_{k}}-1)-h_3-Lu_{k}
\ \text{in}\ \Omega,\\[10pt]
u_{k+1}=0\ \text{on}\ \partial\Omega.
\end{array}
\right.
\end{equation}
The following lemma is proved by mathematical induction.
\begin{lemma}\label{ukbounded}
Assume that the sequence  $\{u_k\}$ satisfies the iterative  equations \eqref{iteraequa}.
Then $u_k$ is uniquely defined for any $k$. Moreover,
\begin{equation*}
0\leq u_k\leq M_0
\end{equation*}
and
\begin{equation*}
0=u_0\leq u_1\leq u_2\leq\cdots\leq u_k\leq\cdots,
\end{equation*}
where $M_0$ satisfies
$\lambda_1e^{AM_0}(e^{AM_0}-1)
+\lambda_2e^{-BM_0}(e^{-BM_0}-1)=4\pi\max\limits_jn_j$.
\end{lemma}

\begin{proof}
First of all, for $k=1$,  equation \eqref{iteraequa} becomes
\begin{equation}\label{1}
\left\{
\begin{array}{lll}
(\Delta-L)u_1=-h_3\ \text{in}\ \Omega,\\[10pt]
u_1=0\ \text{on}\ \partial\Omega.
\end{array}
\right.
\end{equation}
Thus, we obtain the existence and uniqueness of the solution $u_1$ on $\Omega$. Next, we prove
$u_1\leq M_0$. Set $\max\limits_{\bar{\Omega}} u_1=u_1(x_0)$. Then
we have
\begin{equation*}
(\Delta-L)u_1(x_0)\leq-Lu_1(x_0)<0.
\end{equation*}
Therefore, equation \eqref{iteraequa} yields
\begin{equation*}
u_1(x_0)\leq 4\pi n_j/L.
\end{equation*}
Because $L$ is sufficiently large, we know that
\begin{equation*}
u_1\leq M_0.
\end{equation*}
Suppose $0\leq u_k \leq M_0$ and that
$\min\limits_{\bar{\Omega}} u_{k+1}=u_{k+1}(x_0)<0$. Then
\begin{equation}\label{2}
(\Delta-L)u_{k+1}(x_0)\geq-Lu_{k+1}(x_0)>0.
\end{equation}
Since $\lambda_1A-\lambda_2B\ge0$ and
$\lambda_1e^{Au_{k}}(e^{Au_{k}}-1)
+\lambda_2e^{-Bu_{k}}(e^{-Bu_{k}}-1)-Lu_{k}$ is decreasing on $[0,M_0]$ for sufficiently large $L$, equation  \eqref{iteraequa}  gives
\begin{equation*}
\begin{array}{lll}
&\lambda_1e^{Au_{k}}(e^{Au_{k}}-1)
+\lambda_2e^{-Bu_{k}}(e^{-Bu_{k}}-1)-h_3-Lu_{k}
\\[12pt]
&\leq
-h_3(x_0)\\[12pt]
&<0.
\end{array}
\end{equation*}
This contradicts \eqref{2}. Hence, $u_{k+1}\geq0.$

Set $\max\limits_{\bar{\Omega}} u_{k+1}=u_{k+1}(x_0)=N$. Next, we show $N\leq M_0$. Computing the right-hand side of \eqref{iteraequa}, we obtain
\begin{equation}\label{3}
\begin{array}{lll}
&\lambda_1e^{Au_{k}}(e^{Au_{k}}-1)
+\lambda_2e^{-Bu_{k}}(e^{-Bu_{k}}-1)-h_3-Lu_{k}\\[12pt]
&\geq
\lambda_1e^{AM_0}(e^{AM_0}-1)
+\lambda_2e^{-BM_0}(e^{-BM_0}-1)-h_3(x_0)-LM_0.
\end{array}
\end{equation}
On the other hand, from the left-hand side of \eqref{iteraequa}, we obtain
\begin{equation}\label{4}
(\Delta-L)u_{k+1}(x_0)
\leq-Lu_{k+1}(x_0)=-LN.
\end{equation}
Combining \eqref{3} and \eqref{4}, we conclude that
\begin{equation*}
0\leq4\pi\max\limits_jn_j-h_3(x_0)
\leq L(M_0-N).
\end{equation*}
Thus, $N\leq M_0$. Hence, $u_{k+1}\leq M_0$.
It follows that
\begin{equation}\label{5}
0\leq u_k\leq M_0 \quad\text{for any}\quad k\geq1.
\end{equation}
Next, we prove that the sequence $\{u_k\}$ is nondecreasing. Suppose that $0=u_0\leq u_1\leq u_2\leq\cdots\leq u_k$. Since
\begin{equation*}
\lambda_1e^{Au_{k}}(e^{Au_{k}}-1)
+\lambda_2e^{-Bu_{k}}(e^{-Bu_{k}}-1)-h_3-Lu_{k}\\[12pt]
\in l^2(\Omega),
\end{equation*}
there exists a unique solution $u_{k+1}$. On the other hand, combining \eqref{5} with the assumption that $L$ is sufficiently large, we obtain
\begin{equation*}
\begin{array}{lll}
(\Delta-L)(u_{k+1}-u_k)
&=
\lambda_1e^{Au_{k}}(e^{Au_{k}}-1)
+\lambda_2e^{-Bu_{k}}(e^{-Bu_{k}}-1)\\[12pt]
&-
\lambda_1e^{Au_{k-1}}(e^{Au_{k-1}}-1)
-\lambda_2e^{-Bu_{k-1}}(e^{-Bu_{k-1}}-1)\\[12pt]
&-L(u_{k}-u_{k-1})\\[12pt]
&\leq 0.
\end{array}
\end{equation*}
We also have $u_{k+1}-u_k=0$ on $\partial\Omega$.
Hence, it follows from the maximum principle (Lemma \ref{maxprinciple}) that $u_k\leq u_{k+1}$.
This completes the proof of Lemma \ref{ukbounded}.
\end{proof}

From Lemma \ref{ukbounded}, we know that
there exists $u_{\Omega}\in C(\overline{\Omega})$ such that
\begin{equation*}
u_k\rightarrow u_{\Omega}\quad\text{on}\quad\Omega
\end{equation*}
and $u_{\Omega}$ satisfies the equation
\begin{equation}\label{finitesolution}
\left\{
\begin{array}{llll}
\Delta u_{\Omega}
=\lambda_1e^{Au_{\Omega}}(e^{Au_{\Omega}}-1)
+\lambda_2e^{-Bu_{\Omega}}(e^{-Bu_{\Omega}}-1)-h_3
\ \text{in}\ \Omega,\\[10pt]
u_{\Omega}=0\ \text{on}\ \partial\Omega.
\end{array}
\right.
\end{equation}
Let $\{\Omega_k\}_{k=0}^{\infty}$ be a sequence of finite connected subsets of $\mathbb{Z}^n$ such that
\begin{equation*}
\{p_j\}_{j=1}^M\subset\Omega_0\subset\Omega_1
\subset\cdots\subset\Omega_k\subset\cdots,
\qquad
\bigcup_{k=0}^{\infty}\Omega_k=\mathbb{Z}^n.
\end{equation*}

We next pass from the solutions on finite domains to a global solution on $\mathbb Z^n$ by an exhaustion argument.
\begin{lemma}\label{monotoneinZn}
For $\lambda_1A-\lambda_2B\geq0$,
there exists a positive function $\tilde{u}\in l^{\infty}(\mathbb{Z}^n)$ that  satisfies the
Tzitz\'eica  equation,
\begin{equation*}
-\Delta\tilde{u}+\lambda_1e^{A\tilde{u}}(e^{A\tilde{u}}-1)
+\lambda_2e^{-B\tilde{u}}(e^{-B\tilde{u}}-1)-h_3=0
\quad \text{on} \quad \mathbb{Z}^n.
\end{equation*}
\end{lemma}

\begin{proof}
Let $u_{k,i}$ denote the $i$-th iterate associated with the domain $\Omega_k$.
We consider the following problem.
\begin{equation}\label{iterzn}
\left\{
\begin{array}{llll}
(\Delta-L)u_{k,i}
=\lambda_1e^{Au_{k,i-1}}(e^{Au_{k,i-1}}-1)
+\lambda_2e^{-Bu_{k,i-1}}(e^{-Bu_{k,i-1}}-1)-h_3-Lu_{k,i-1}
\ \text{in}\ \Omega_k,\\[10pt]
u_{k,i}=0\ \text{on}\ \partial\Omega_k.
\end{array}
\right.
\end{equation}
By Lemma \ref{ukbounded}, we have
\begin{equation}\label{6}
0\leq u_{k,i}\leq u_{k,i+1}\leq M_0.
\end{equation}
Moreover, $u_{k,i}\rightarrow u_k$ as $i\rightarrow+\infty$, and $u_k$ satisfies the equation \eqref{finitesolution}.
Next we show that
\begin{equation*}
u_{k+1,i+1}\geq u_{k,i+1}.
\end{equation*}
First, we observe that
\begin{equation}\label{7}
\left\{
\begin{array}{lll}
(\Delta-L)u_{k,1}=-h_3\ \text{in}\ \Omega_k,\\[10pt]
u_{k,1}=0\ \text{on}\ \partial\Omega_k,
\end{array}
\right.
\end{equation}
and
\begin{equation}\label{8}
\left\{
\begin{array}{lll}
(\Delta-L)u_{k+1,1}=-h_3\ \text{in}\ \Omega_k,\\[10pt]
u_{k+1,1}\geq0\ \text{on}\ \partial\Omega_k.
\end{array}
\right.
\end{equation}
Combining \eqref{7} with \eqref{8} yields
\begin{equation}\label{9}
\left\{
\begin{array}{lll}
(\Delta-L)(u_{k+1,1}-u_{k,1})=0\ \text{in}\ \Omega_k,\\[10pt]
u_{k+1,1}-u_{k,1}\geq0\ \text{on}\ \partial\Omega_k.
\end{array}
\right.
\end{equation}
By the maximum principle, we obtain
\begin{equation*}
u_{k+1,1}-u_{k,1}\geq0 \quad\text{in} \quad\Omega_k.
\end{equation*}
Suppose that
$u_{k+1,i}-u_{k,i}\geq0$ in $\Omega_k$.
Consider the following two equations:
\begin{equation}\label{10}
\left\{
\begin{array}{llll}
(\Delta-L)u_{k+1,i+1}
=\lambda_1e^{Au_{k+1,i}}(e^{Au_{k+1,i}}-1)
+\lambda_2e^{-Bu_{k+1,i}}(e^{-Bu_{k+1,i}}-1)-h_3-Lu_{k+1,i}
\ \text{in}\ \Omega_k,\\[10pt]
u_{k+1,i+1}\geq0\ \text{on}\ \partial\Omega_k,
\end{array}
\right.
\end{equation}
and
\begin{equation}\label{11}
\left\{
\begin{array}{llll}
(\Delta-L)u_{k,i+1}
=\lambda_1e^{Au_{k,i}}(e^{Au_{k,i}}-1)
+\lambda_2e^{-Bu_{k,i}}(e^{-Bu_{k,i}}-1)-h_3-Lu_{k,i}
\ \text{in}\ \Omega_k,\\[10pt]
u_{k,i+1}=0\ \text{on}\ \partial\Omega_k.
\end{array}
\right.
\end{equation}
Set $f(x)=\lambda_1e^{Ax}(e^{Ax}-1)
+\lambda_2e^{-Bx}(e^{-Bx}-1)-Lx$.
Since $L$ is sufficiently large, the function $f(x)$ is decreasing on $[0,M_0]$.
It follows from \eqref{10} and \eqref{11} that
\begin{equation}\label{12}
\left\{
\begin{array}{lll}
(\Delta-L)(u_{k+1,i+1}-u_{k,i+1})
=f(u_{k+1,i})-f(u_{k,i})\leq0\ \text{in}\ \Omega_k,\\[10pt]
u_{k+1,i+1}-u_{k,i+1}\geq0\ \text{on}\ \partial\Omega_k.
\end{array}
\right.
\end{equation}
Therefore, we have $u_{k+1,i+1}\geq u_{k,i+1}$. Hence,
\begin{equation*}
u_{k+1,i+1}\geq u_{k,i+1}\quad\text{on}\quad\Omega_k.
\end{equation*}
Letting $i\rightarrow+\infty$, we obtain
\begin{equation*}
u_{k+1}\geq u_k\quad\text{on}\quad\Omega_k.
\end{equation*}
Set
\begin{equation}\label{3.16}
\tilde{u}_{k}=
\left\{
\begin{array}{lll}
u_k,\quad\text{in}\quad\Omega_k,\\[12pt]
0,\ \ \quad\text{in}\quad\mathbb{Z}^n\backslash\Omega_k.
\end{array}
\right.
\end{equation}
Then
\begin{equation}\label{monotoznuk}
\tilde{u}_{k+1}\geq\tilde{u}_{k}\quad\text{on}\quad
\mathbb{Z}^n.
\end{equation}
Furthermore, it follows from \eqref{5} that
$0\leq \tilde{u}_{k}\leq M_0$. Consequently,
\begin{equation*}
\tilde{u}_{k}\rightarrow \tilde{u}
\end{equation*}
as $k\rightarrow+\infty$, and $\tilde{u}$ satisfies the Tzitz\'eica  equation,
\begin{equation}\label{solutionzn}
-\Delta \tilde{u}+\lambda_1e^{A\tilde{u}}(e^{A\tilde{u}}-1)
+\lambda_2e^{-B\tilde{u}}(e^{-B\tilde{u}}-1)-h_3=0
\quad \text{on} \quad \mathbb{Z}^n.
\end{equation}
\end{proof}
Having established the existence of a bounded nonnegative solution, we next derive its exponential decay estimate.
\begin{lemma}\label{positivedecayestimate}
For $\lambda_1A-\lambda_2B>0$, the function $\tilde{u}$ in Lemma \ref{monotoneinZn}  has the
decay estimate
\begin{equation*}
\tilde{u}=O(e^{-md(x)}),
\end{equation*}
where $m=\log\left(
1+\frac{\lambda_1A-\lambda_2B}{2n}
\right)$.
\end{lemma}

\begin{proof}
First, we show that $\tilde{u}_{k}$, defined in \eqref{3.16}, is a subsolution of \eqref{solutionzn}.\\
\noindent
(1) If $x\in\Omega_k\setminus \Omega_0$, $y\sim x$ and $y\in\Omega_k\cup\partial\Omega_k$, then we
have
\begin{equation}\label{15}
\Delta\tilde{u}_k(x)
=
\lambda_1e^{A\tilde{u}_k(x)}(e^{A\tilde{u}_k(x)}-1)
+\lambda_2e^{-B\tilde{u}_k(x)}(e^{-B\tilde{u}_k(x)}-1).
\end{equation}
(2) If $x\in\mathbb{Z}^n
\backslash(\Omega_k\cup\partial\Omega_k), y\sim x$ and $y\in\mathbb{Z}^n\backslash\Omega_k$,
then we have $\tilde{u}_k(x)=\tilde{u}_k(y)=0$.
Further, we obtain
\begin{equation}\label{16}
\Delta\tilde{u}_k(x)
=
\lambda_1e^{A\tilde{u}_k(x)}(e^{A\tilde{u}_k(x)}-1)
+\lambda_2e^{-B\tilde{u}_k(x)}(e^{-B\tilde{u}_k(x)}-1)
=0.
\end{equation}
(3) If $x\in\partial\Omega_k,y\sim x$ and $y\in\mathbb{Z}^n$, then we have $\tilde{u}_k(x)=0$.
Moreover,
\begin{equation}\label{17}
\Delta\tilde{u}_k(x)=\sum_{y\sim x}
\left(\tilde{u}_k(y)-\tilde{u}_k(x)\right)
=\sum_{y\sim x}\tilde{u}_k(y)\geq0
\end{equation}
and
\begin{equation}\label{18}
\lambda_1e^{A\tilde{u}_k(x)}(e^{A\tilde{u}_k(x)}-1)
+\lambda_2e^{-B\tilde{u}_k(x)}(e^{-B\tilde{u}_k(x)}-1)
=0.
\end{equation}
Combining \eqref{15}--\eqref{18}, we deduce that
\begin{equation*}
\left\{
\begin{array}{lll}
\Delta\tilde{u}_k(x)
\geq
\lambda_1e^{A\tilde{u}_k(x)}(e^{A\tilde{u}_k(x)}-1)
+\lambda_2e^{-B\tilde{u}_k(x)}(e^{-B\tilde{u}_k(x)}-1)
\quad\text{in}\quad\mathbb{Z}^n\backslash\bar{\Omega}_0,\\[12pt]
\tilde{u}_k(x)=0,\quad d(x)\rightarrow\infty.
\end{array}
\right.
\end{equation*}

Next, it is easy to see that
the function $f(x)=(\lambda_1e^{Ax}(e^{Ax}-1)
+\lambda_2e^{-Bx}(e^{-Bx}-1))/x$ is increasing
for $x>0$. For $x\in \mathbb{Z}^n\backslash\bar{\Omega}_0$, a direct computation shows that
\begin{equation}\label{13}
\begin{aligned}
\displaystyle{\Delta \tilde{u}_{k}}
&\geq\displaystyle{
\lambda_1e^{A\tilde{u}_{k}}(e^{A\tilde{u}_{k}}-1)
+\lambda_2e^{-B\tilde{u}_{k}}(e^{-B\tilde{u}_{k}}-1)
}\\[12pt]
&\geq\lim\limits_{\varepsilon\rightarrow 0^+}
\lambda_1e^{A(\tilde{u}_{k}+\varepsilon)}(e^{A(\tilde{u}_{k}+\varepsilon)}-1)
+\lambda_2e^{-B(\tilde{u}_{k}+\varepsilon)}(e^{-B(\tilde{u}_{k}+\varepsilon)}-1)\\[12pt]
&\geq\lim\limits_{\varepsilon\rightarrow 0^+}\left(\frac{
\lambda_1e^{A(\tilde{u}_{k}+\varepsilon)}(e^{A(\tilde{u}_{k}+\varepsilon)}-1)
+\lambda_2e^{-B(\tilde{u}_{k}+\varepsilon)}(e^{-B(\tilde{u}_{k}+\varepsilon)}-1)}{\tilde{u}_{k}+\varepsilon}
\cdot (\tilde{u}_{k}+\varepsilon)\right)\\[12pt]
&\geq\lim\limits_{\varepsilon\rightarrow 0^+}
\left(\frac{\lambda_1e^{A\varepsilon}(e^{A\varepsilon-1})
+\lambda_2e^{-B\varepsilon}(e^{-B\varepsilon}-1)}{\varepsilon}\cdot (\tilde{u}_{k}+\varepsilon)\right)\\[12pt]
&=
(\lambda_1A-\lambda_2B)\tilde{u}_{k}.
\end{aligned}
\end{equation}
Denote the barrier function by
\begin{equation*}
\phi(x)=-e^{-md(x)},
\end{equation*}
where $m=\log\left(
1+\frac{\lambda_1A-\lambda_2B}{2n}
\right)$.
By the same computation as in \cite[Page 496]{MR5045709}, we have
\begin{equation*}
\begin{array}{lll}
\Delta \phi(x)
&\geq n\left(e^{-m}+e^{m}-2\right)\phi(x)\\[10pt]
&\geq2n\left(
e^{m}-1
\right)\phi(x)\\[12pt]
&=(\lambda_1A-\lambda_2B)\phi(x).
\end{array}
\end{equation*}
We consider the subset
\begin{equation*}
\Omega_0'=\{x\in\mathbb{Z}^n:d(x)\geq R_1\}
\end{equation*}
and $\Omega_0'\cap\bar{\Omega}_0=\emptyset$. Choosing a sufficiently large constant $C$,
we have
\begin{equation*}
(\Delta-(\lambda_1A-\lambda_2B))(C\phi(x)+\tilde{u}_k(x))\geq0\ \ \text{on}\ \ \Omega_0'
\end{equation*}
and
\begin{equation*}
C\phi(x)+\tilde{u}_k(x)\leq0
\quad\text{if}\quad d(x)=R_1,
\end{equation*}
together with
\begin{equation*}
\lim\limits_{|x|\rightarrow+\infty}
\left(C\phi(x)+\tilde{u}_k(x)\right)=0.
\end{equation*}
Applying Lemma \ref{maxprinciple}, we obtain
\begin{equation*}
C\phi(x)+\tilde{u}_k(x)\leq0
\end{equation*}
on $\Omega'_0$. That is,
\begin{equation*}
0\leq\tilde{u}_k(x)\leq Ce^{-md(x)},
\end{equation*}
where $C$  is independent of  $k$.
Passing to the limit as $k\rightarrow+\infty$, we obtain
\begin{equation*}
0\leq \tilde{u}(x)\leq Ce^{-md(x)},
\end{equation*}
where $\tilde{u}$ satisfies \eqref{solutionzn}. This completes the proof of Lemma \ref{positivedecayestimate}.
\end{proof}

\section{The proof of Theorem \ref{negativethm2}}
For $\lambda_1A-\lambda_2B>0$, let $L>0$ be some sufficiently large constant  and set $u_0=0$. Then we consider the iterative equation as follows:
\begin{equation}\label{iteraequanegetive}
\left\{
\begin{array}{llll}
(\Delta-L)u_k
=\lambda_1e^{Au_{k-1}}(e^{Au_{k-1}}-1)
+\lambda_2e^{-Bu_{k-1}}(e^{-Bu_{k-1}}-1)+h_3-Lu_{k-1}
\ \text{on}\ \Omega,\\[10pt]
u_k=0\ \text{on}\ \partial\Omega.
\end{array}
\right.
\end{equation}Denote $$f(u):=\frac{\lambda_1e^{A{u}}(e^{A{u}}-1)
+\lambda_2e^{-B{u}}(e^{-B{u}}-1)}{u}.$$
It is easy to see that $f(u)$ is increasing in $(-\infty,0)$ and
\begin{equation*}
\lim\limits_{u\rightarrow 0^-}f(u)=\lambda_1A-\lambda_2B>0.
\end{equation*}
Then $f(s_0)=0$ with $s_0$ defined in Theorem \ref{negativethm2}.

We first establish the monotonicity and uniform boundedness of the iterative sequence.
\begin{lemma}\label{monolemnegative}
Under the assumption of Theorem \ref{negativethm2}, for any $\delta\in (0,-s_0)$,
let $u_k$ satisfy \eqref{iteraequanegetive}. Then for each $k$, $u_k$ is uniquely defined and
\begin{equation*}
0=u_0\geq u_1\geq u_2\geq\cdots.
\end{equation*}
Moreover, $s_0+\delta\leq u_{k}\leq0$.
\end{lemma}

\begin{proof}
We prove the lemma by induction as follows.
For $k=1$, equation \eqref{iteraequanegetive} becomes
\begin{equation}
\left\{
\begin{array}{lll}
(\Delta-L)u_1=h_3\ \text{on}\ \Omega,\\[10pt]
u_1=0\ \text{on}\ \partial\Omega.
\end{array}
\right.
\end{equation}
Thus, we obtain the existence and uniqueness of the solution $u_1$ on $\Omega$.
It follows from Lemma \ref{maxprinciple} that $u_1\leq0$ on $\bar{\Omega}$. Suppose that $0=u_0\geq u_1\geq u_2\geq \cdots \geq u_k$. Since
\begin{equation*}
\lambda_1e^{Au_{k}}(e^{Au_{k}}-1)
+\lambda_2e^{-Bu_{k}}(e^{-Bu_{k}}-1)+h_3-Lu_{k}\in l^2(\Omega),
\end{equation*}
it follows that there exists a unique solution
$u_{k+1}$ to \eqref{iteraequanegetive}. Furthermore, we have
\begin{equation*}
\begin{array}{lll}
(\Delta-L)(u_{k+1}-u_k)
&=
\lambda_1e^{Au_k}(e^{Au_k}-1)-\lambda_1e^{Au_{k-1}}(e^{Au_{k-1}}-1)\\[10pt]
&\quad+\lambda_2e^{-Bu_k}(e^{-Bu_k}-1)-\lambda_2e^{-Bu_{k-1}}(e^{-Bu_{k-1}}-1)
-L(u_k-u_{k-1})\\[10pt]
&=
\lambda_1(e^{2Au_k}-e^{2Au_{k-1}})-\lambda_1(e^{Au_k}-e^{Au_{k-1}})
+\lambda_2(e^{-2Bu_k}-e^{-2Bu_{k-1}})\\[10pt]
&\quad -\lambda_2(e^{-Bu_k}-e^{-Bu_{k-1}})-L(u_k-u_{k-1})\\[10pt]
&\geq
(2\lambda_1Ae^{2A\xi_1}-L)(u_k-u_{k-1})\\[12pt]
&\geq 0,
\end{array}
\end{equation*}
with $u_k\leq\xi_1\leq u_{k-1}\leq 0$, where $-\lambda_1(e^{Au_k}-e^{Au_{k-1}})\geq 0$ due to $u_k\leq u_{k-1}\leq 0$, and $$\lambda_2(e^{-2Bu_k}-e^{-2Bu_{k-1}})
-\lambda_2(e^{-Bu_k}-e^{-Bu_{k-1}})\geq 0,$$ since the function $e^{-2Bx}-e^{-Bx}$ is decreasing for $x\leq 0$.
Then  Lemma \ref{maxprinciple} implies that $u_{k+1}\leq u_k$ and we  obtain the monotonicity of $u_k$.

Next, we prove the boundedness of $u_k$.
Note that $u_k$ satisfies the following
equation
\begin{equation*}
\Delta{u_k}=\lambda_1e^{A{u_k}}(e^{A{u_k}}-1)
+\lambda_2e^{-B{u_k}}(e^{-B{u_k}}-1)
\quad\text{on}\quad\Omega\setminus{\Omega}_0.
\end{equation*}
 If $u_k\in[s_0+\delta,0)$ for any $\delta\in(0,-s_0)$, then
\begin{equation*}
f(u_k)\geq C=f(s_0+\delta).
\end{equation*}
Next, we prove that $u_k\geq s_0+\delta$.  We claim that $u_k\in[s_0+\delta,0)$. Firstly, we have
\begin{equation*}
\left\{
\begin{array}{lll}
(\Delta-L)u_1=h_3\ \text{on}\ \Omega,\\[10pt]
u_1=0\ \text{on}\ \partial\Omega.
\end{array}
\right.
\end{equation*}
Then $u_1\leq0$. Assume that
$\min\limits_{\bar{\Omega}} u_1=u_1(x_0)=-N$. A direct calculation yields
\begin{equation*}
h_3(x_0)=(\Delta-L)u_1(x_0)\geq-Lu_1(x_0)=LN.
\end{equation*}
This implies
\begin{equation*}
N\leq\frac{4\pi n_j}{L}.
\end{equation*}
Since $L$ is sufficiently large,  we have $N\leq -(s_0+\delta)$. Suppose $s_0+\delta\leq u_k\leq
0$.  We already know $u_{k+1}\leq0$
from the above proof.
Set $u_{k+1}(x_0)=\min\limits_{\bar{\Omega}}u_{k+1}(x)=-N$ and we have
\begin{equation*}
(\Delta-L)u_{k+1}(x_0)\geq-Lu_{k+1}(x_0)=LN.
\end{equation*}
According to the equation \eqref{iteraequanegetive}, we obtain
\begin{equation*}
\begin{array}{lll}
LN
&\leq
\lambda_1e^{Au_{k}(x_0)}(e^{Au_{k}(x_0)}-1)
+\lambda_2e^{-Bu_{k}(x_0)}(e^{-Bu_{k}(x_0)}-1)+h_3(x_0)-Lu_{k}(x_0)\\[12pt]
&\leq
\lambda_1e^{A(s_0+\delta)}(e^{A(s_0+\delta)}-1)
+\lambda_2e^{-B(s_0+\delta)}(e^{-B(s_0+\delta)}-1)+h_3(x_0)-L(s_0+\delta).
\end{array}
\end{equation*}
The last inequality follows from the fact that
$\lambda_1e^{Ax}(e^{Ax}-1)
+\lambda_2e^{-Bx}(e^{-Bx}-1)-Lx$ is decreasing on $[s_0+\delta,0]$ if
 $L$ is sufficiently large. Therefore, we have
\begin{equation*}
N\leq-(s_0+\delta)+\frac{1}{L}
\left(
\lambda_1e^{A(s_0+\delta)}(e^{A(s_0+\delta)}-1)
+\lambda_2e^{-B(s_0+\delta)}(e^{-B(s_0+\delta)}-1)+h_3(x_0)
\right)
\end{equation*}
Since $\lambda_1e^{A(s_0+\delta)}(e^{A(s_0+\delta)}-1)
+\lambda_2e^{-B(s_0+\delta)}(e^{-B(s_0+\delta)}-1)\leq0$ and  $4\pi\max{n_j}\leq-g(s_0+\delta)$, we obtain $N\leq -(s_0+\delta)$. Then, $u_{k+1}\geq (s_0+\delta)$. Consequently, we get 
\begin{equation*}
s_0+\delta\leq u_{k}<0.
\end{equation*}
Thus, we get the desired conclusion.

\end{proof}

From Lemma \ref{monolemnegative}, we can obtain the following lemma directly.

\begin{lemma}\label{finitenegativesolution}
Let the sequence $\{u_k\}$ be defined by \eqref{iteraequanegetive}. Then there  exists
$u_{\Omega}\in C(\bar{\Omega})$ such that
\begin{equation*}
u_k\rightarrow u_{\Omega} \quad\text{as}\quad k\rightarrow+\infty
\end{equation*}
on $\bar{\Omega}$ and $u_{\Omega}$ solves the following problem
\begin{equation}\label{19}
\left\{
\begin{array}{llll}
\Delta u_{\Omega}
=\lambda_1e^{Au_{\Omega}}(e^{Au_{\Omega}}-1)
+\lambda_2e^{-Bu_{\Omega}}(e^{-Bu_{\Omega}}-1)+h_3
\ \text{on}\ \Omega,\\[10pt]
u_{\Omega}=0\ \text{on}\ \partial\Omega.
\end{array}
\right.
\end{equation}
\end{lemma}

We next establish a comparison result for the solutions obtained by the above iteration scheme.

\begin{lemma}\label{solutionmax}
Let $\Omega\subset\mathbb{Z}^n$ be a finite subset and let $\{u_k\}$ satisfy \eqref{iteraequanegetive}. For any $w\in C(\bar{\Omega})$ satisfying
\begin{equation*}
\left\{
\begin{array}{lll}
&\Delta w\geq \lambda_1e^{Aw}(e^{Aw}-1)
+\lambda_2e^{-Bw}(e^{-Bw}-1)+h_3\ \  \text{on}\ \ \Omega\\[10pt]
&w(x)\leq0\ \  \text{on}\ \  \partial\Omega,
\end{array}
\right.
\end{equation*}
we have
\begin{equation*}
  0=u_0\geq u_1\geq u_2\geq\cdots\geq u_k\geq\cdots\geq u_{\Omega}\geq w.
\end{equation*}
\end{lemma}

\begin{proof}
It is easy to see that  the function
$F_1(x)=\lambda_1e^{Ax}(e^{Ax}-1)
+\lambda_2e^{-Bx}(e^{-Bx}-1)$ is increasing for $x\in(0,\infty)$. Firstly, we claim
 $\sup\limits_{x\in\Omega}w(x)\leq0$.
 Otherwise, we can choose $x_0\in\Omega$ such that $w(x_0)=\sup_{x\in\Omega}w(x)>0$,
which gives
 \begin{equation*}
 0\geq \Delta w(x_0)\geq \lambda_1e^{Aw}(e^{Aw}-1)
+\lambda_2e^{-Bw}(e^{-Bw}-1)+h_3>0.
 \end{equation*}
This is a contradiction and the claim is proved.

We assume $w\leq u_k\leq0$ and calculate
\begin{equation*}
\begin{aligned}
&(\Delta-L)(u_{k+1}-w)\\[10pt]
\leq&
\lambda_1(e^{2Au_k}-e^{2Aw}-e^{Au_k}+e^{Aw})
+\lambda_2(e^{-2Bu_k}-e^{-Bu_k}-e^{-2Bw}+e^{-Bw})-L(u_k-w)\\[10pt]
\leq& \lambda_1(2Ae^{2A\xi_1}(u_k-w)
-Ae^{\xi_2}(u_k-w))-L(u_k-w)\\[10pt]
\leq& (3A\lambda_1-L)(u_k-w)\\[10pt]
\leq& 0.
\end{aligned}
\end{equation*}
Then  Lemma \ref{maxprinciple} implies that
$u_{k+1}\geq w$. Hence,  we get the desired conclusion.
\end{proof}
Let $\{\Omega_i\}_{i=0}^{\infty}$ be a sequence of finite connected subsets of $\mathbb Z^n$, which satisfies
\begin{equation*}
\Omega_0\subset\Omega_1
\subset\cdots\subset\Omega_k\subset\cdots,\  \bigcup_{i=1}^{\infty}\Omega_i=\mathbb{Z}^n.
\end{equation*}
For convenience, we denote $u^i:=u_{\Omega_i}$.
For any $1\leq j\leq k$ with  $\Omega_j\subset\Omega_k$,  Lemma \ref{solutionmax} yields
\begin{equation*}
u^k\leq u^j \quad\text{on}\quad\Omega_j.
\end{equation*}
Define $\widetilde u^k$ by
\begin{equation*}
\tilde{u}^k=
\left\{
\begin{array}{lll}
u^k,\quad\text{in}\quad\Omega_k,\\[12pt]
0,\quad \ \text{in}\quad\mathbb{Z}^n\backslash\Omega_k.
\end{array}
\right.
\end{equation*}
Then we obtain
\begin{equation}\label{monoznnegative}
0\geq\widetilde{u}^1\geq\widetilde{u}^2\geq\cdots
\geq\widetilde{u}^k\geq\cdots \quad\text{on}\quad
\mathbb{Z}^n.
\end{equation}

\begin{lemma}\label{negativesolutionznlemm}
Let $\{\tilde{u}^k\}$ be the sequence defined in \eqref{monoznnegative}. Then there exists a function ${u}\in C(\mathbb{Z}^n)$ such that $\tilde{u}^k\rightarrow{u}$ as $k\rightarrow+\infty$ and ${u}$ satisfies
\begin{equation}\label{04}
-\Delta{u}+\lambda_1e^{A{u}}(e^{A{u}}-1)
+\lambda_2e^{-B{u}}(e^{-B{u}}-1)+h_3=0
\quad \text{on} \quad \mathbb{Z}^n.
\end{equation}
\end{lemma}

\begin{proof}
From Lemma \ref{monolemnegative}, we have
$|\widetilde{u}^k|\leq C$, where $C$ is independent of $k$. Moreover, by \eqref{monoznnegative}, there exists a function
${u}\in C(\mathbb{Z}^n)$ such that
 $\tilde{u}^k\rightarrow{u}$ as $k\rightarrow+\infty$ and ${u}$ satisfies
\begin{equation*}
-\Delta{u}+\lambda_1e^{A{u}}(e^{A{u}}-1)
+\lambda_2e^{-B{u}}(e^{-B{u}}-1)+h_3=0
\quad \text{on} \quad \mathbb{Z}^n.
\end{equation*}

\end{proof}
It remains to establish the decay estimate for the solution obtained above.

\begin{lemma}\label{negativeestimate}
The solution $u$ of equation \eqref{tz2} satisfies the decay estimate
\begin{equation*}
u=O(e^{-m(1-\epsilon_0)d(x)}),
\end{equation*}
where $m=\log\left(
1+\frac{\lambda_1A-\lambda_2B}{2n}
\right)$ and $0<\epsilon_0<1$.
\end{lemma}

\begin{proof}
As in the derivation of \eqref{15}--\eqref{18}, we can show that $\tilde{u}^k$ satisfies
\begin{equation}
-\Delta\tilde{u}^k+\lambda_1e^{A\tilde{u}^k}
(e^{A\tilde{u}^k}-1)
+\lambda_2e^{-B\tilde{u}^k}(e^{-B\tilde{u}^k}-1)+h_3\geq 0
\quad \text{on} \quad \mathbb{Z}^n.
\end{equation}
Moreover,
according to Lemma \ref{monolemnegative}, we know that
\begin{equation*}
s_0+\delta\leq \tilde{u}^k\leq0.
\end{equation*}
For $x\in \mathbb{Z}^n\backslash\bar{\Omega}_0$,
a direct computation shows that
\begin{equation*}
\begin{aligned}
\displaystyle{\Delta \tilde{u}^{k}}
&\leq\displaystyle{
\lambda_1e^{A\tilde{u}^{k}}(e^{A\tilde{u}^{k}}-1)
+\lambda_2e^{-B\tilde{u}^{k}}(e^{-B\tilde{u}^{k}}-1)
}\\[12pt]
&\leq\lim\limits_{\varepsilon\rightarrow 0^+}
\lambda_1e^{A(\tilde{u}^{k}+\varepsilon)}(e^{A(\tilde{u}^{k}+\varepsilon)}-1)
+\lambda_2e^{-B(\tilde{u}^{k}+\varepsilon)}(e^{-B(\tilde{u}^{k}+\varepsilon)}-1)\\[12pt]
&\leq\lim\limits_{\varepsilon\rightarrow 0^+}\left(\frac{
\lambda_1e^{A(\tilde{u}^{k}+\varepsilon)}(e^{A(\tilde{u}^{k}+\varepsilon)}-1)
+\lambda_2e^{-B(\tilde{u}^{k}+\varepsilon)}(e^{-B(\tilde{u}^{k}+\varepsilon)}-1)}{\tilde{u}^{k}+\varepsilon}
\cdot (\tilde{u}^{k}+\varepsilon)\right)\\[12pt]
&\leq\lim\limits_{\varepsilon\rightarrow 0^+}
\left(f(\tilde{u}^k+\varepsilon)\cdot (\tilde{u}^{k}+\varepsilon)\right)\\[12pt]
&=
f(\tilde{u}^k)\tilde{u}^{k}\leq f(s_0+\delta)\tilde{u}^{k}.
\end{aligned}
\end{equation*}

Next, note that $0<f(s_0+\delta)< \lim_{u\rightarrow0^-} f(u)=\lambda_1A-\lambda_2B$. Then
there exists a constant $\epsilon_0(\delta)\in(0,1)$ such that
$ f(s_0+\delta)\geq 2n\left[\left(1+\frac{\lambda_1A-\lambda_2B}{2n}\right)^{1-\epsilon_0}-1 \right]:=c_1$.
Then we have
\begin{equation*}
\Delta\tilde{u}^k
\leq c_1\tilde{u}^k
\quad\text{in}\quad\mathbb{Z}^n\backslash\bar{\Omega}_0.
\end{equation*}
As in  Lemma \ref{positivedecayestimate},
we set
\begin{equation*}
\phi(x)=-e^{-m(1-\epsilon_0)d(x)}
\end{equation*}
and also get
\begin{equation*}
\Delta \phi(x)\geq  c_1\phi(x).
\end{equation*}
Moreover, we have
\begin{equation*}
(\Delta-c_1)(C(\epsilon_0)\phi(x)-\tilde{u}^k(x))\geq0\ \ \text{on}\ \ \Omega_0'
\end{equation*}
and
\begin{equation*}
\lim\limits_{|x|\rightarrow+\infty}
\left(C(\epsilon_0)\phi(x)-\tilde{u}^k(x)\right)=0,
\quad
C(\epsilon_0)\phi(x)-\tilde{u}^k(x)\leq0
\quad\text{if}\quad d(x)=R_1,
\end{equation*}
where
$\Omega_0'=\{x\in\mathbb{Z}^n:d(x)\geq R_1\}$. By Lemma \ref{maxprinciple}, the above estimate yields
\begin{equation}\label{21}
0\geq \tilde{u}^k(x)
\geq C(\epsilon_0)\phi(x)= -C(\epsilon_0)e^{-m(1-\epsilon_0)d(x)}.
\end{equation}
By convergence and \eqref{21}, we deduce that
\begin{equation*}
0\geq {u}(x)
\geq -C(\epsilon_0)e^{-m(1-\epsilon_0)d(x)},
\end{equation*}
where ${u}(x)$ is  the solution of \eqref{04}.
\end{proof}

\section{The proof of Theorem \ref{positivetz2thm1} and Theorem \ref{nagetivetz2tm2}}
Throughout this section, we denote the Banach space
$$c_0(\mathbb Z^n):=\left\{u:\mathbb Z^n\to\mathbb R:\lim_{|x|\to\infty}u(x)=0
\right\},$$
endowed with the norm
$$\|u\|_{\infty}:=\sup_{x\in\mathbb Z^n}|u(x)|.$$
We first establish the following existence result for the linear problem $\Delta-L$ on $\mathbb Z^n$.

\begin{lemma}\label{l5.1}
Let $L>0$ and $f\in c_0(\mathbb Z^n)$.
Then there exists a unique function $u\in c_0(\mathbb Z^n)$ satisfying
$$
(\Delta-L)u=f
\qquad\text{on }\mathbb\quad Z^n,
$$
with
$$
\|u\|_{\infty}\le \frac{1}{L}\|f\|_{\infty}.
$$
Moreover, if $f\le 0$ on $\mathbb Z^n$, then $u\ge 0$ on $\mathbb Z^n$. Similarly, if
$f\ge 0$, then $u\le 0$.
\end{lemma}

\begin{proof}
For $u:\mathbb Z^n\to\mathbb R$, recall that
$$
\Delta u(x)
=
\sum_{y\sim x}\bigl(u(y)-u(x)\bigr)
=
\sum_{y\sim x}u(y)-2n\,u(x).
$$
Define the operator $\mathcal A$ by
$$
\mathcal A u(x):=\sum_{y\sim x}u(y).
$$
Then
$$
\Delta=\mathcal A-2nI.
$$

We first note that $\mathcal A$ maps $c_0(\mathbb Z^n)$ into itself. Indeed, if
$u\in c_0(\mathbb Z^n)$ and $|x|\to\infty$, then for every neighbor $y\sim x$
we also have $|y|\to\infty$. Since each $x\in\mathbb Z^n$ has exactly $2n$
neighbors,
$$
\mathcal A u(x)=\sum_{y\sim x}u(y)\longrightarrow 0
\qquad\text{as }|x|\to\infty.
$$
Moreover,
$$
|\mathcal A u(x)|
\le
\sum_{y\sim x}|u(y)|
\le
2n\|u\|_{\infty}.
$$
Therefore,
$$
\|\mathcal A\|_{c_0\to c_0}\le 2n.
$$
Now, the equation
$$
(\Delta-L)u=f\qquad\text{on }\mathbb\quad Z^n
$$
can be rewritten as
$$
\bigl((2n+L)I-\mathcal A\bigr)u=-f\qquad\text{on }\mathbb\quad Z^n,
$$
or equivalently,
$$
\left(I-\frac{\mathcal A}{2n+L}\right)u
=
-\frac{1}{2n+L}f\qquad\text{on }\mathbb\quad Z^n.
$$
Since $L>0$,
$$
\left\|
\frac{\mathcal A}{2n+L}
\right\|_{c_0\to c_0}
\le
\frac{2n}{2n+L}
<1.
$$
Hence the Neumann series converges in the operator norm on
$c_0(\mathbb Z^n)$, and
$$
\left(I-\frac{\mathcal A}{2n+L}\right)^{-1}
=
\sum_{k=0}^{\infty}
\left(\frac{\mathcal A}{2n+L}\right)^k.
$$
Consequently,
$$
u
=
-\frac{1}{2n+L}
\sum_{k=0}^{\infty}
\left(\frac{\mathcal A}{2n+L}\right)^k f
\in c_0(\mathbb Z^n),
$$
which proves existence. And the uniqueness follows directly from the maximum principle.

The same representation gives
\begin{align*}
\|u\|_{\infty}
\le
\frac{1}{2n+L}
\sum_{k=0}^{\infty}
\left(\frac{2n}{2n+L}\right)^k
\|f\|_{\infty}=
\frac{1}{L}\|f\|_{\infty}.
\end{align*}


Finally, suppose that $f\le0$. Since the operator $\mathcal A$ preserves
nonnegativity, namely
$$
v\ge0 \quad\Longrightarrow\quad \mathcal A v\ge0,
$$
the representation formula
$$
u
=
-\frac{1}{2n+L}
\sum_{k=0}^{\infty}
\left(\frac{\mathcal A}{2n+L}\right)^k f
$$
implies $u\ge0$. The case $f\ge0$ follows analogously.
\end{proof}

\subsection{The proof of Theorem \ref{positivetz2thm1} }

Choose a sufficiently large constant $L>1$ and set $u_0=0$. We consider the following iterative  equations,
\begin{equation}\label{iteraequapositivetz2}
\left\{
\begin{array}{llll}
(\Delta-L)u_{k+1}
=\lambda_1(e^{Au_{k}}-1)
+\lambda_2(e^{-Bu_{k}}-1)-h_3-Lu_{k}
\ \text{on}\ \mathbb{Z}^n,\\[10pt]
\lim\limits_{|x|\rightarrow+\infty} u_{k+1}(x)=0.
\end{array}
\right.
\end{equation}
To state the following lemma, let $M_0>0$ satisfy
\begin{equation}\label{5.2}\lambda_1(e^{AM_0}-1)
+\lambda_2(e^{-BM_0}-1)=4\pi\max\limits_j n_j,\end{equation}
which can be achieved since $\lambda_1A-\lambda_2B>0$.
\begin{lemma}\label{positivemonotonetz2}
Let the sequence $\{u_k\}$ satisfy \eqref{iteraequapositivetz2}. Then each $u_k$ is uniquely defined
and
\begin{equation*}
0\leq u_1\leq u_2\leq\cdots\leq u_k\leq M_0,
\end{equation*}
where $M_0$ is given in \eqref{5.2}.
\end{lemma}

\begin{proof}
First, we know that $u_1$ satisfies
\begin{equation}\label{01}
\left\{
\begin{array}{lll}
(\Delta-L)u_1=-h_3,\\[10pt]
u_1\rightarrow 0\ \text{as}\ |x|\rightarrow+\infty,
\end{array}
\right.
\end{equation}which, by Lemma \ref{l5.1},
implies that $u_1\geq0$. Set $\max_{\mathbb{Z}^n}u_1=u_1(x_0)>0$. Then \eqref{01} gives
\begin{equation*}
-4\pi\max_j n_j\leq(\Delta-L)u_1(x_0)\leq-Lu_1(x_0).
\end{equation*}
It follows that  $u_1(x_0)\leq 4\pi\max_j n_j/L$. Now, choose a constant $L_0>0$ such that for $L\geq L_0$, we obtain $0\leq u_1\leq M_0$.
Next, we can choose a sufficiently large $L_1>\max(1,L_0)$, such that the function $\lambda_1(e^{Ax}-1)
+\lambda_2(e^{-Bx}-1)-Lx$ is decreasing on $[0,M_0]$ for any $L\geq L_1$. In the following part of this proof, we always assume $L\geq L_1$.

To continue, it follows from Lemma \ref{l5.1} that there exists a unique $u_2\geq 0$ satisfying \eqref{iteraequapositivetz2} for $k=1$. Moreover, we can show that $u_2\leq M_0$. Indeed, set $\max\limits_{\mathbb{Z}^n}u_{2}(x)
=u_{2}(x_1)=N$ for some $x_1\in \mathbb {Z}^n$. By the decreasing property of the function $\lambda_1(e^{Ax}-1)
+\lambda_2(e^{-Bx}-1)-Lx$ on $[0,M_0]$ and \eqref{5.2}, we deduce that
\begin{equation}\label{02}
\begin{array}{lll}
&\lambda_1(e^{Au_{1}(x_1)}-1)
+\lambda_2(e^{-Bu_{1}(x_1)}-1)-h_3(x_1)-Lu_{1}(x_1)\\[12pt]
&\geq
\lambda_1(e^{AM_0}-1)
+\lambda_2(e^{-BM_0}-1)-LM_0-h_3(x_1)\\[12pt]
&=4\pi\max\limits_j n_j-LM_0-h_3(x_1).
\end{array}
\end{equation}
The maximum principle gives that
\begin{equation}\label{03}
(\Delta-L)u_{2}(x_1)\leq-Lu_{2}(x_1)=-LN.
\end{equation}
Combining \eqref{02} and \eqref{03} leads to
\begin{equation*}
L(M_0-N)\geq 4\pi\max\limits_j n_j-h_3(x_1)\geq0,
\end{equation*}
which implies $\max\limits_{\mathbb{Z}^n}u_{2}(x)=N\leq M_0$. Inductively, we obtain
\begin{equation}
0\leq u_k\leq M_0, \text{ for any }k\in \mathbb N.
\end{equation}
On the other hand, by induction, since $u_0=0\leq u_1$, we assume $0\leq u_1\leq u_2\leq\cdots\leq u_k$. We compute
\begin{equation*}
\begin{aligned}
(\Delta-L)(u_{k+1}-u_k)
=&
\lambda_1(e^{Au_{k}}-1)+\lambda_2(e^{-Bu_{k}}-1)\\[12pt]
&-\lambda_1(e^{Au_{k-1}}-1)-\lambda_2(e^{-Bu_{k-1}}-1)
-L(u_k-u_{k-1})\\[12pt]
\leq&0
\end{aligned}
\end{equation*}
for sufficiently large $L$. Moreover, $u_k(x)-u_{k-1}(x)=0$ as $|x|\rightarrow+\infty$. Then, we have
$u_k\leq u_{k+1}$.
Hence, we obtain Lemma \ref{positivemonotonetz2}.
\end{proof}
By the Monotone Convergence Theorem,
there exists $\tilde u$ such that $0\le\tilde u\le M_0$ and $u_k\rightarrow \tilde{u}$ pointwise on $\mathbb Z^n$ and $\tilde{u}$ satisfies the equation \eqref{tz3}.
Using an argument similar to that in the proof of Lemma \ref{negativeestimate}, we can show  the exponential decay of the nonnegative solution $\tilde{u}$ obtained in Lemma \ref{positivemonotonetz2}.
In fact, we can prove the stronger statement that any bounded nonnegative solution of \eqref{tz3} decays exponentially as $d(x)\to\infty$.

\begin{lemma}\label{l5.3}
Let $\tilde{u}\in l^\infty(\mathbb Z^n)$ be a nonnegative solution of \eqref{tz3},
and assume that $\lambda_1A-\lambda_2B>0$.
Then $\tilde u$ has the following decay estimate:
$$
 \tilde{u}(x)
=
O\bigl(e^{-m d(x)}\bigr)
\quad\text{as}\quad d(x)\to\infty,
$$
where
$$
m
=
\log\left(
1+\frac{\lambda_1A-\lambda_2B}{2n}
\right).
$$
\end{lemma}

\begin{proof}
Set
$$
\mu:=\lambda_1A-\lambda_2B>0
$$
and
$$
G(t)
:=
\lambda_1(e^{At}-1)
+
\lambda_2(e^{-Bt}-1).
$$
Choose a finite subset $\Omega_0\subset\mathbb Z^n$ such that
$$
\operatorname{supp}(h_3)
=
\{p_1,\ldots,p_M\}
\subset\Omega_0.
$$
Then the equation \eqref{tz3} yields
\begin{equation}\label{5.7}
\Delta\widetilde u(x)
=
G(\widetilde u(x)),
\quad
x\in\mathbb Z^n\setminus\Omega_0.
\end{equation}
Since $\widetilde u\ge0$, for every $t\ge0$ we have
$$
e^{At}-1\ge At, \quad e^{-Bt}-1\ge -Bt.
$$
Therefore,
$$
\begin{aligned}
G(t)
&=
\lambda_1(e^{At}-1)
+
\lambda_2(e^{-Bt}-1)\\[12pt]
&\ge
\lambda_1At-\lambda_2Bt\\[12pt]
&=
\mu t.
\end{aligned}
$$
Combining this inequality with \eqref{5.7}, we obtain
\begin{equation}\label{5.8}
\Delta\widetilde u(x)
\ge
\mu\widetilde u(x),
\quad
x\in\mathbb Z^n\setminus\Omega_0.
\end{equation}
Recall that
$$
\Delta\widetilde u(x)
=
\sum_{y\sim x}
\bigl(\widetilde u(y)-\widetilde u(x)\bigr).
$$
Thus \eqref{5.8} implies
\begin{equation}\label{5.9}
\sum_{y\sim x}\widetilde u(y)
\ge
(2n+\mu)\widetilde u(x)
\quad
\text{for}\quad x\in\mathbb Z^n\setminus\Omega_0.
\end{equation}
It follows from \eqref{5.9} that, for every
$x\in\mathbb Z^n\setminus\Omega_0$, there exists at least one neighbor
$y\sim x$ such that
\begin{equation}\label{5.10}
\widetilde u(y)
\ge
\left(1+\frac{\mu}{2n}\right)\widetilde u(x).
\end{equation}
Set
$$
q:=1+\frac{\mu}{2n}>1.
$$
Fix $x\in\mathbb Z^n\setminus\Omega_0$. If $\widetilde u(x)=0$, the desired
estimate is trivial. Therefore, we assume that $\widetilde u(x)>0$.
Starting from $x_0:=x$, as long as $x_k\notin\Omega_0$, choose a neighbor
$x_{k+1}\sim x_k$ such that
$$
\widetilde u(x_{k+1})
\ge
q\,\widetilde u(x_k).
$$
By induction,
\begin{equation}\label{5.11}
\widetilde u(x_k)
\ge
q^k\widetilde u(x)
\end{equation}
for every $k$ for which $x_0,\ldots,x_{k-1}\notin\Omega_0$.

We claim that this sequence must enter $\Omega_0$ after finitely many
steps. Otherwise, \eqref{5.11} would hold for every
$k\ge1$, and hence
$$
\widetilde u(x_k)
\ge
q^k\widetilde u(x)
\longrightarrow+\infty
\quad\text{as}\quad k\to\infty,
$$
because $q>1$ and $\widetilde u(x)>0$. This contradicts the assumption
$\widetilde u\in l^\infty(\mathbb Z^n)$. Therefore, there exists an
integer $N\ge1$ such that
$$
x_N\in\Omega_0\quad \text{and}\quad
x_0,\ldots,x_{N-1}\notin\Omega_0.
$$
Since $x_0,x_1,\ldots,x_N$ is a nearest-neighbor path joining $x$ to
$\Omega_0$, its length $N$ is at least the graph distance from $x$ to
$\Omega_0$. Hence
$$N\ge d(x,\Omega_0):=
\min_{z\in\Omega_0}d(x,z).$$
Using \eqref{5.11}, we obtain
$$
\sup_{z\in\Omega_0}\widetilde u(z)
\ge
\widetilde u(x_N)
\ge
q^N\widetilde u(x)
\ge
q^{d(x,\Omega_0)}\widetilde u(x).
$$
Consequently,
\begin{equation}\label{5.12}
\widetilde u(x)
\le
\left(
\sup_{z\in\Omega_0}\widetilde u(z)
\right)
q^{-d(x,\Omega_0)}.
\end{equation}
Since $\Omega_0$ is finite, there exists $R_0>0$ such that
$$
\Omega_0\subset\{z\in\mathbb Z^n:d(z)\le R_0\}.
$$
For every $x\in\mathbb Z^n$ and $z\in\Omega_0$, the triangle inequality
gives
$$
d(x,z)\ge d(x)-d(z)\ge d(x)-R_0.
$$
Therefore,
$$
d(x,\Omega_0)\ge d(x)-R_0.
$$
It follows from \eqref{5.12} that
$$
\begin{aligned}
\widetilde u(x)
&\le
\left(
\sup_{z\in\Omega_0}\widetilde u(z)
\right)
q^{-d(x,\Omega_0)}\\
&\le
\left(
\sup_{z\in\Omega_0}\widetilde u(z)
\right)
q^{R_0}q^{-d(x)}\\
&\le
C e^{-m d(x)},
\end{aligned}
$$
where
$$
m
=
\log q
=
\log\left(
1+\frac{\mu}{2n}
\right)
=
\log\left(
1+\frac{\lambda_1A-\lambda_2B}{2n}
\right)
$$
and
$$
C
=
q^{R_0}
\sup_{z\in\Omega_0}\widetilde u(z).
$$
This completes the proof.
\end{proof}
It follows from Lemmas \ref{positivemonotonetz2} and \ref{l5.3} that there exists a solution $\widetilde u$ of equation \eqref{tz3} satisfying
the desired estimates in Theorem \ref{positivetz2thm1}. We now show the uniqueness.
If there exists another bounded nonnegative solution
$v$ to the equation \eqref{tz3}, then we compute
\begin{equation*}
\Delta(\tilde{u}-v)
=(\lambda_1Ae^{A\xi_1}
-\lambda_2Be^{-B\xi_2})(\tilde{u}-v),
\end{equation*}
where $\xi_1,\xi_2\in [0,\max(\tilde{u},v)]$. Note that
$$\lambda_1Ae^{A\xi_1}
-\lambda_2Be^{-B\xi_2}\geq \lambda_1A-\lambda_2B>0,$$ and $v$ decays exponentially
as $d(x)\rightarrow \infty$ by Lemma \ref{l5.3}.
Then, we conclude that this solution is unique by the maximum principle. It remains to show the comparison relation
$$
u\leq  \tilde{u},
$$
where $u$ is the nonnegative solution obtained in Theorem \ref{positivethm1} and
$ \tilde{u}$ is the nonnegative solution of equation \eqref{tz3}.

Set
$$
F_1(t)
:=
\lambda_1e^{At}(e^{At}-1)
+
\lambda_2e^{-Bt}(e^{-Bt}-1)
$$
and
$$
F_2(t)
:=
\lambda_1(e^{At}-1)
+
\lambda_2(e^{-Bt}-1).
$$
A direct computation gives
$$
F_1(t)-F_2(t)
=
\lambda_1(e^{At}-1)^2
+
\lambda_2(e^{-Bt}-1)^2
\geq 0.
$$
Moreover, for $t\geq 0$,
$$
F_2'(t)
=
\lambda_1Ae^{At}
-
\lambda_2Be^{-Bt}
\geq
\lambda_1A-\lambda_2B
>0.
$$
Hence $F_2$ is strictly increasing on $[0,+\infty)$.
Let
$$
w:=u-\widetilde u.
$$
Since both $u$ and $\widetilde u$ decay to zero at infinity, we have
$$
w(x)\to 0
\qquad\text{as }|x|\to\infty.
$$
Suppose, for contradiction, that
$$
\sup_{\mathbb Z^n} w>0.
$$
Then $w$ attains a positive maximum at some point $x_0\in\mathbb Z^n$.
Therefore,
$$
\Delta w(x_0)\leq 0.
$$
On the other hand, since $u$ and $\widetilde u$ satisfy equations
\eqref{tz1} and \eqref{tz3}, respectively, we have
$$
\Delta w
=
F_1(u)-F_2(\widetilde u).
$$
Since
$$
u(x_0)>\widetilde u(x_0)\geq 0,
$$
it follows that
$$
\begin{aligned}
\Delta w(x_0)
&=
F_1(u(x_0))-F_2(\widetilde u(x_0))\\
&\geq
F_2(u(x_0))-F_2(\widetilde u(x_0))\\
&>0,
\end{aligned}
$$
which contradicts $\Delta w(x_0)\leq 0$. Consequently,
$$
u\leq \widetilde u
\qquad\text{on }\mathbb Z^n.
$$
Therefore,
$$
0\leq u\leq \widetilde u.
$$

\subsection{The proof of Theorem \ref{nagetivetz2tm2} }

Choose a sufficiently large constant $L>0$ and set $u_0=0$. We consider the following iterative  equations,
\begin{equation*}
\left\{
\begin{array}{llll}
(\Delta-L)u_{k+1}
=\lambda_1(e^{Au_{k}}-1)
+\lambda_2(e^{-Bu_{k}}-1)+h_3-Lu_{k}
\ \text{on}\ \mathbb{Z}^n,\\[10pt]
\lim\limits_{|x|\rightarrow+\infty} u_{k+1}(x)=0.
\end{array}
\right.
\end{equation*}
We claim that
\begin{equation*}
0=u_0\geq u_1\geq u_2\geq\cdots\geq u_k\geq\cdots\geq u,
\end{equation*}
where $u$
is the solution obtained in Theorem \ref{negativethm2}. For $k=0$, we have
\begin{equation*}
\left\{
\begin{array}{lll}
(\Delta-L)u_1=h_3,\\[12pt]
\lim\limits_{|x|\rightarrow+\infty} u_{1}(x)=0.
\end{array}
\right.
\end{equation*}
This implies $u_1\leq0$. Let $u$ be the nonpositive solution obtained in Theorem 1.2. Suppose
that
$$
u\leq u_k\leq 0.
$$
We prove that
$$
u\leq u_{k+1}.
$$
Recall that
$$
F_1(t)
:=
\lambda_1e^{At}(e^{At}-1)
+
\lambda_2e^{-Bt}(e^{-Bt}-1)
$$
and
$$
F_2(t)
:=
\lambda_1(e^{At}-1)
+
\lambda_2(e^{-Bt}-1).
$$
Since $u$ satisfies equation (1.10), we have
$$
\Delta u=F_1(u)+h_3.
$$
On the other hand, the iterative equation gives
$$
(\Delta-L)u_{k+1}
=
F_2(u_k)+h_3-Lu_k.
$$
Therefore,
$$
\begin{aligned}
(\Delta-L)(u-u_{k+1})
&=
\Delta u-Lu-(\Delta-L)u_{k+1}\\
&=
F_1(u)-F_2(u_k)-L(u-u_k).
\end{aligned}
$$

Now define
$$
H(t):=F_2(t)-Lt.
$$
Since all the relevant functions take values in the fixed interval
$[s_0+\delta,0]$, we may choose $L>0$ sufficiently large such that
$H$ is decreasing on this interval. Since
$$
u\leq u_k,
$$
we obtain
$$
H(u)\geq H(u_k).
$$
Furthermore,
$$
F_1(u)-F_2(u)
=
\lambda_1(e^{Au}-1)^2
+
\lambda_2(e^{-Bu}-1)^2
\geq0.
$$
Consequently,
$$
\begin{aligned}
(\Delta-L)(u-u_{k+1})
&=
F_1(u)-F_2(u_k)-L(u-u_k)\\
&=
\bigl(F_1(u)-F_2(u)\bigr)
+
\bigl(H(u)-H(u_k)\bigr)\\
&\geq0.
\end{aligned}
$$
Since
$$
u(x)-u_{k+1}(x)\to0
\qquad\text{as }|x|\to\infty,
$$
the maximum principle yields
$$
u-u_{k+1}\leq0.
$$
Hence,
$$
u\leq u_{k+1}.
$$
 Suppose
$u_1\geq u_2\geq\cdots\geq u_k\geq u$. Next, we show $u_k\geq u_{k+1}$.
As in the calculation above, we have
\begin{equation*}
\begin{aligned}
(\Delta-L)(u_{k+1}-u_k)
=&
\lambda_1(e^{Au_{k}}-1)
+\lambda_2(e^{-Bu_{k}}-1)\\[12pt]
&-
\lambda_1(e^{Au_{k-1}}-1)
-\lambda_2(e^{-Bu_{k-1}}-1)
-L(u_{k}-u_{k-1})\\[12pt]
\geq& 0
\end{aligned}
\end{equation*}
for sufficiently large $L>0$. This implies $u_{k+1}\leq u_k$.
Consequently, we have
\begin{equation*}
0=u_0\geq u_1\geq u_2\geq\cdots\geq u_k\geq\cdots\geq u.
\end{equation*}
Since
\begin{equation*}
u=O(e^{-m(1-\epsilon_0)d(x)}),
\end{equation*}
there exists $\bar{u}$ such that
\begin{equation*}
u_k\rightarrow\bar{u}
\end{equation*}
as $k\rightarrow+\infty$. Furthermore, we have
\begin{equation*}
\bar{u}=O(e^{-m(1-\epsilon_0)d(x)})
\end{equation*}
and $\bar{u}$ satisfies \eqref{tz4}.

\section*{Acknowledgements}
Both authors gratefully acknowledge financial support from the China Scholarship Council (CSC).\\

\noindent\textbf{Conflicts of interest/Competing interests}: Not applicable.\\
\textbf{Data availability statement}: Data sharing is not applicable to this article as no datasets were generated or analysed during the current study.\\

\normalem\bibliographystyle{plain}{}

\end{document}